\documentclass[11pt]{article}
\usepackage{amsmath}
\usepackage{amsthm}
\usepackage{amssymb,amsfonts}
\usepackage{hyperref}
\usepackage{latexsym}
\usepackage{todonotes}
\usepackage{nicefrac}
\usepackage{epsfig}
\usepackage{stmaryrd}
\usepackage{setspace}
\usepackage{enumerate}
\usepackage[all]{xypic}
\usepackage{bbm,ifpdf,tikz}
\ifpdf
\usepackage{pdfsync}
\fi

\newtheorem{theorem}{Theorem}[section]

\newtheorem{lemma}[theorem]{Lemma}

\newtheorem{corollary}[theorem]{Corollary}

{
\theoremstyle{definition}

\newtheorem{remark}[theorem]{Remark}
}

\newcommand{\excise}[1]{}

\renewcommand{\dim}{\operatorname{dim}}

\renewcommand{\and}{\qquad\text{and}\qquad}

\newcommand{\Hom}{\operatorname{Hom}}

\newcommand{\Q}{\mathbb{Q}}

\newcommand{\R}{\mathbb{R}}

\newcommand{\cA}{\mathcal{A}}
\newcommand{\la}{\lambda}

\newcommand{\FS}{\operatorname{FS}}

\newcommand{\FSop}{\operatorname{FS^{op}}}

\begin{document}
\spacing{1.2}
\noindent{\Large\bf Stability phenomena for resonance arrangements}\\

\noindent{\bf Nicholas Proudfoot\footnote{Supported by NSF grants DMS-1565036 and DMS-1954050.} and Eric Ramos\footnote{Supported by NSF grant DMS-1704811.}}\\
Department of Mathematics, University of Oregon,
Eugene, OR 97403\\

{\small
\begin{quote}
\noindent {\em Abstract.}
We prove that the $i^\text{th}$ graded pieces of the Orlik--Solomon algebras or Cordovil algebras of resonance arrangements
form a finitely generated $\FSop$-module, thus obtaining information about the growth of their dimensions 
and restrictions on the irreducible representations of symmetric groups that they contain.
\end{quote} }

\section{Introduction}
Let $\cA(n)$ be the collection of all hyperplanes in $\R^n$ that are perpendicular to some nonzero vector with entries in the set $\{0,1\}$.
This hyperplane arrangement is called the {\bf resonance arrangement} of rank $n$.  The resonance arrangement has connections
to algebraic geometry, representation theory, geometric topology, mathematical physics, and economics; 
for a survey of these connections, see \cite[Section 1]{Kuhne}.
Of particular interest is the set of chambers of $\cA(n)$.  Amazingly, despite the simplicity of the definition, no formula for the number
of chambers as a function of $n$ is known.  A more refined invariant of $\cA(n)$ is its characteristic polynomial,
whose coefficients (after taking absolute values) have sum equal to the number of chambers.  K\"uhne
has made some progress toward understanding the coefficient of $t^{n-i}$ in the characteristic polynomial as a function of $n$
with $i$ fixed.  Our purpose is to shed a new light on K\"uhne's result, to generalize it to a wider class of arrangements, and to study
the action of the symmetric group $\Sigma_n$ on various algebraic invariants of these arrangements.

Let $S\subset \R$ be any finite set, and let $\cA_S(n)$
be the collection of hyperplanes that are perpendicular to a nonzero vector with entries in $S$.  If $S=\{0,1\}$, $\cA_S(n)$ is the resonance
arrangement.  If $S = \{\pm 1\}$, it is the {\bf threshold arrangement}, which is studied in \cite{GMP}.  For each positive integer $d$,
let $M_S(n,d)$ denote 
the set of $n$-tuples of vectors in $\R^d$ such that no nontrivial\footnote{Nontrivial
means that, if $0\in S$, we do not allow all coefficients to be 0.} linear combination of all $n$ vectors with coefficients in $S$ is equal to zero. 
The cohomology ring of $M_S(n,d)$ is generated in degree $d-1$ \cite[Corollary 5.6]{dLS}.  If $d$ is even,
the presentation of this ring in \cite{dLS} coincides with that of the
{\bf Orlik--Solomon algebra} of $\cA_S(n)$ (with all degrees multiplied by $d-1$) \cite{OS}.  If $d$ is odd and greater than $1$, then it coincides
with that of the {\bf Cordovil algebra} of $\cA_S(n)$ (with all degrees multiplied by $d-1$) \cite{Co}; see also \cite[Example 5.8]{Moseley}.\footnote{For
$d$ odd, the presentation in \cite{dLS} incorrectly omits the relations that each of the generators squares to zero.}
In particular, for any $n\geq 1$, $d\geq 2$, and $i\geq 0$, the dimension
$b_S^i(n) = \dim H^{(d-1)i}\big(M_S(n,d); \Q\big)$ is equal to $(-1)^i$ times the coefficient of $t^{n-i}$ in the characteristic polynomial of $\cA_S(n)$.

These vector spaces carry more information than just their dimension; they also carry actions of the symmetric group $\Sigma_n$,
which acts by permuting the $n$ vectors.  These representations are isomorphic for all even $d\geq 2$ and for all odd $d\geq 3$,
but the $d=2$ and $d=3$ cases are genuinely different.
The total cohomology $H^*\big(M_S(n,3); \Q\big)$ with all degrees combined
is isomorphic as a representation of $\Sigma_n$ to $H^0\big(M_S(n,1); \Q\big)$, which is
the permutation representation with basis indexed by the chambers of $\cA_S(n)$ \cite[Theorem 1.4(b)]{Moseley}.

For fixed $S\subset \R$, $d\geq 2$, and $i\geq 0$, we will define in the next section a contravariant module $B_S^{i,d}$
over the category of finite sets with surjections that takes the set $[n]$ to $H^{(d-1)i}\big(M_S(n,d); \Q\big)$.

\begin{theorem}\label{main}
The module $B_S^{i,d}$ is finitely generated in degrees $\leq |S|^i$.
\end{theorem}

Combining Theorem \ref{main} with \cite[Theorem 4.1]{fs-braid}, we obtain the following numerical results:\footnote{The deepest of these statements, namely
the fact that the dimension generating function for a finitely generated $\FSop$-module is rational with prescribed poles, 
is due to Sam and Snowden \cite[Corollary 8.1.4]{sam}.}
\begin{corollary} Fix a finite set $S\subset\R$ and a pair of integers $d\geq 2$ and $i\geq 0$.
\begin{itemize} 
\item[1.] The generating function $$G_{\!S}^i(t) := \sum_{n=1}^\infty b_S^i(n) t^n$$
is a rational function with poles contained in the set $\{1/j\mid 1\leq j \leq |S|^i\}$, with at worst a simple pole at $|S|^{-i}$.  
Equivalently, there exist polynomials
$\{c_S^{i,j}(n)\mid 1\leq j \leq |S|^i\}$ such that, for $n$ sufficiently large, $$b_S^i(n) = \sum_{j=1}^{|S|^i} c_S^{i,j}(n) j^n,$$
and the last polynomial $c_S^{i,|S|^i}(n)$ is a constant polynomial.
\item[2.] For any partition $\la$ of $n$, let $V_{\la}$ denote the irreducible representation of $\Sigma_n$ indexed by $\la$.
If $\Hom_{\Sigma_n}\!\Big(V_\la, H^{(d-1)i}\big(M_S(n,d); \Q\big)\Big)\neq 0$, then $\la$ has at most $|S|^i$ rows.
\item[3.] For any partition $\la$ with $n \geq |\la| + \la_1$, let $\la(n)$ be the {\bf padded partition} of $n$ obtained from $\la$ by adding
a row of length $n-|\la|$.  
For any $\la$, the function $$n\mapsto \dim\Hom_{\Sigma_n}\!\left(V_{\la(n)}, H^{(d-1)i}\big(M_S(n,d); \Q\big)\right)$$ 
is bounded above by a polynomial in $n$.
In particular, if $\la$ is the empty partition, this says that the multiplicity of the trivial representation in $H^{(d-1)i}\big(M_S(n,d); \Q\big)$ 
is bounded above by a polynomial in $n$.
\end{itemize}
\end{corollary}

\begin{remark}
A stronger version of item (1) above for the resonance arrangement appears in \cite[Theorem 1.4]{Kuhne}.   
K\"uhne proves that the polynomials $c_{\{0,1\}}^{i,j}(n)$ are {\em all}
constant (i.e. that all poles of $G_{\{0,1\}}^i(t)$ are simple), obtains bounds on their sizes, 
and shows that the equality holds for all $n$, not just sufficiently large $n$ (i.e. that the limit as $t$ goes to $\infty$ of $G_{\{0,1\}}^i(t)$ is zero).  
It should be possible to categorify K\"uhne's theorem by proving that the restriction of $B_{\{0,1\}}^{i,d}$ to the category of {\bf ordered surjections} \cite{sam} is isomorphic to a direct sum of shifts of principal projectives, with summands indexed by K\"uhne's {\bf functional prototypes}.
The cost of working with ordered surjections would be that we would lose all information about the action of the symmetric group.
\end{remark}

\noindent
{\em Acknowledgments:}
The authors are grateful to Lou Billera for telling them about the arrangement $\cA(n)$ and about K\"uhne's work.

\section{The proof}
Let $\FS$ denote the category whose objects are nonempty finite sets and whose morphisms are surjective maps.
An {\bf \boldmath{$\FSop$}-module} over $\Q$ is a contravariant functor from $\FS$ to the category of rational vector spaces.
For each finite set $F$, we have the {\bf principal projective} module $P_F$, which sends a finite set $E$ to the vector space
with basis $\Hom_{\FS}(E,F)$, with morphisms defined on basis elements by composition.  An $\FSop$-module $N$ is said
to be {\bf finitely generated} if it is a quotient of a finite sum $\oplus_i P_{F_i}$ of principal projectives, and it is said to be
{\bf finitely generated in degrees \boldmath{$\leq m$}} if the sets $F_i$ can all be taken to have cardinality less than or equal to $m$.
This is equivalent to saying that, for all $E$,
the vector space $N(E)$ finite dimensional and is spanned by the images of the pullbacks along various maps $\varphi:E\to F$, 
where $F$ has cardinality less than or equal to $m$.

\begin{lemma}\label{tensor}
Suppose that $N_1$ is finitely generated in degrees $\leq m_1$ and $N_2$ is finitely generated in degrees $\leq m_2$.
Then the pointwise tensor product $N_1\otimes N_2$ is finitely generated in degrees $\leq m_1m_2$.
\end{lemma}

\begin{proof}
We immediately reduce to the case where $N_1 = P_{[m_1]}$ and $N_2 = P_{[m_2]}$.
For any $\varphi:E\to[m]$, let $e_{\varphi}$ denote the corresponding basis element of $P_{[m]}(E)$.
Then $N_1\otimes N_2$ has basis $$\left\{e_{\varphi_1}\otimes e_{\varphi_2}\mid \varphi_1:E\to [m_1], \varphi_2:E\to [m_2]\right\}.$$
Given the pair of surjections $(\varphi_1,\varphi_2)$, let $F\subset [m_1]\times[m_2]$ denote the image of $\varphi_1\times\varphi_2$,
let $\varphi = \varphi_1\times\varphi_2 \in\Hom_{\FS}(E,F)$,
and let $\psi_1:F\to[m_1]$ and $\psi_2:F\to[m_2]$ denote the coordinate projections.
It is clear that we have $e_{\varphi_1}\otimes e_{\varphi_2} = \varphi^*(e_{\psi_1}\otimes e_{\psi_2})$.  Since the cardinality
of $F$ is at most $m_1m_2$, this completes the proof.
\end{proof}

Fix a positive integer $d$ and a finite set $S\subset\R$.
To any finite set $E$, we associated the space $M_S(E, d)$ of $E$-tuples of vectors in $\R^d$ such that 
any nontrivial linear combination of the vectors with coefficients in $S$ is nonzero.
Given a surjection $\varphi:E\to F$, we obtain a map $$\varphi_*:M_S(E,d)\to M_S(F,d)$$
by adding the vectors in each fiber of $\varphi$.
These maps define a functor from $\FS$ to the category of topological spaces.  By taking rational cohomology in degree $(d-1)i$,
we obtain an $\FSop$-module $B_S^{i,d}$.  We prove the following theorem, which implies the three statements in the introduction.

\begin{proof}[Proof of Theorem \ref{main}.]
As noted above, the cohomology of $M_S(E,d)$ is generated as an algebra in degree $d-1$, hence $B_S^{i,d}$
is a quotient of $(B_S^{1,d})^{\otimes i}$.  By Lemma \ref{tensor}, this means that it is sufficient to prove that $B_S^{1,d}$
is finitely generated in degrees $\leq |S|$.  For any finite set $F$, the vector space $B_S^{1,d}(F)$ has a generating set indexed by
nonzero elements of $S^F$ \cite[Corollary 5.6]{dLS} (these generators form a basis unless two nonzero elements of $S^F$ are proportional, in which case the corresponding generators are equal).  For any nonzero $v\in S^F$, let $x_v\in B_S^{1,d}(F)$ be the corresponding generator.
Concretely, if we take $x\in H^{d-1}(\R^d\smallsetminus\{0\}; \Q)$ to be the standard generator, then $x_v$ is equal to the pullback of $x$
along the map $$f_v: M_S(F,d)\to \R^d\smallsetminus\{0\}$$ 
that sends an $F$-tuple of vectors to its linear combination with coefficients determined by $v$.
Given a surjection $\varphi:E\to F$, we have $f_v\circ \varphi_* = f_{\varphi^*v}$, and therefore
$$\varphi^*(x_v) = \varphi^*\circ f_v^*(x) =  f_{\varphi^*v}^*(x) = 
x_{\varphi^*v}\in B_S^{1,d}(E).$$
Since every element of $S^E$ may be pulled back from a subset of cardinality at most $|S|$, 
$B_S^{1,d}$ is generated in degrees $\leq |S|$.
\end{proof}

\begin{remark}
Our construction also works if we replace $\R$ with an arbitrary field $k$ and we take $S$ to be a finite subset of $k$.
We define the arrangement $\cA_{k,S}(n)$ in $k^n$ as above, we denote its complement by $M_{k,S}(E,1)$, and we take
$B_{k,S}^{i,1}(E)$ to be the \'etale cohomology group 
$H_{\text{\'et}}^{i}\big(M_{k,S}(E,1)\otimes_k\bar{k}; \Q_l\big)$ for some prime $l$ not equal to the characteristic of $k$, 
which is isomorphic to the degree $i$ part of the Orlik--Solomon algebra of $\cA_{k,S}(n)$.  This is an $\FSop$-module over $\Q_l$,
and the same argument shows that it is finitely generated in degrees $\leq |S|^i$.  

An interesting special case is where
$k=\mathbb{F}_q$ is a finite field and $S=k$, so that our arrangement $\cA_{{\mathbb{F}_q},{\mathbb{F}_q}}(n)$ is the collection
of all hyperplanes in $\mathbb{F}_q^n$.  This arrangement has characteristic polynomial $(t-1)(t-q)\cdots(t-q^{n-1})$,
and therefore the $i^\text{th}$ Betti number is equal to the evaluation of the $i^\text{th}$ elementary symmetric polynomial at the values
$1,q,\ldots,q^{n-1}$.  This implies that the Hilbert series of our module is
$$q^{\binom{i}{2}}t^i\prod_{j=0}^{i}\frac{1}{1-q^j t}\, ,$$ which has simple poles at $q^{-j}$ for $j=0,1,\ldots,i$.
The projectivization of $M_{{\mathbb{F}_q},{\mathbb{F}_q}}(n,1)\otimes_{\mathbb{F}_q}\bar{\mathbb{F}}_q$ 
is a Deligne--Lusztig variety for the group $\operatorname{GL}_n(\mathbb{F}_q)$.  
\end{remark}

\bibliography{./symplectic}
\bibliographystyle{amsalpha}

\end{document}